\documentclass[12pt]{article}
\usepackage[utf8]{inputenc}
\usepackage{amsmath}
\usepackage{amssymb}
\usepackage{mathrsfs}
\usepackage{amsthm}
\usepackage{cite}
\usepackage{hyperref}
\usepackage{enumerate}
\usepackage{combelow}
\usepackage[left=0.7in,right=0.7in,top=0.7in,bottom=0.7in]{geometry}
\usepackage{titlesec}
\titleformat*{\section}{\large\bfseries}
\newtheorem{theorem}{Theorem}[section]
\newtheorem{lemma}[theorem]{Lemma}

\newtheorem{definition}[theorem]{Definition}
\newtheorem{example}[theorem]{Example}

\numberwithin{equation}{section}
\title{Properties $(BR)$ and $(BgR)$ for Bounded Linear Operators
}
\author{\large Anuradha Gupta$^1$ and Ankit Kumar$^2$\\ $^1$ {\small Department of Mathematics, Delhi College of Arts and Commerce,}\\ {\small University of Delhi, Netaji Nagar, New Delhi-110023, India.}\\{\small E-mail:dishna2@yahoo.in}\\{\small $^2$Department of Mathematics, University of Delhi, New Delhi-110007, India.}\\{\small E-mail:1995ankit13@gmail.com}}
\date{}
\begin{document}
\maketitle
\begin{abstract}
In this paper we introduce the notion of property $(BR)$ and property $(BgR)$ for bounded linear operators defined on an infinite-dimensional Banach space. These properties in connection with Weyl type theorems and  in the frame of polaroid operators are investigated. Moreover, we study the stability of these properties under perturbations by   commuting finite-dimensional, quasi-nilpotent, Riesz and algebraic operators.
\end{abstract}

\section{Introduction and Preliminaries}
Throughout this paper, $X$ denotes an infinite-dimensional Banach space and let $B(X)$ be the Banach algebra of all bounded linear operators on $X$. For $T \in B(X),$ we denote the null space, range of $T$, adjoint of $T$ by $N(T), T(X)$ and $T^{*}$, respectively. Let $\sigma(T), \sigma_{a}(T),$ iso $\sigma(T)$ and iso $\sigma_{a}(T)$ denote the spectrum of $T$, approximate point spectrum of $T$, isolated points of spectrum and isolated points of approximate point spectrum, respectively. Let $\alpha(T)$= dim $N(T)$ and $\beta(T)$= codim $T(X)$ be the nullity of $T$ and deficiency of $T,$ respectively. An operator $T \in B(X)$ is called an upper semi-Fredholm operator if $\alpha(T) < \infty $ and $T(X)$ is closed . An operator $T \in B(X)$ is called a lower semi-Fredholm operator if $\beta(T) < \infty $. The  class of all upper semi-Fredholm operators (resp. lower semi-Fredholm operators) is denoted by $\phi_{+} (X)$ (resp. $\phi_{-}(X)$).   Let $\phi_\pm(X):=\phi_{+}(X)\cup \phi_{-}(X)$ be the class of all semi-Fredholm operators. For $T\in \phi_\pm(X)$ the index of $T$ is defined by ind $(T$):= $\alpha(T)-\beta(T)$. The class of all Fredholm operators is defined by $\phi(X):=\phi_{+}(X)\cap \phi_{-}(X)$.  The class of all upper semi-Weyl operators (resp. lower semi-Weyl operators) is defined by $W_{+} (X)=\{T\in\phi_{+} (X):$ ind $(T) \leq 0\} $ (resp. $W_{-} (X)=\{T\in\phi_{-} (X):$ ind $(T) \geq 0\}.$ The set of all Weyl operators is defined by $W(X):=W_{+}(X) \cap W_{-}(X)=\{T\in \phi(X):$ ind $(T)=0\}$.
The \emph{upper semi-Weyl spectrum} is defined by \[
\sigma_{uw}(T):= \{ \lambda \in \mathbb{C}: \lambda I-T \notin W_{+}(X)\}.
\]
The \emph{upper semi-Weyl spectrum} is defined by  \[ \sigma_{lw}(T):=\{ \lambda \in \mathbb{C}: \lambda I-T \notin W_{-}(X)\}
\]
and the \emph{Weyl spectrum} is defined by  \[ \sigma_w(T):=\{\lambda \in \mathbb{C}: \lambda I-T \notin W(X)\}.\]
For $T \in B(X)$, $p(T)$ be the ascent of $T$ defined as  the smallest non negative integer $p$ such that $N(T^{p})= N(T^{p+1})$. If  no such integer exists, then we set $p(T)=\infty$. Similarly,  $q(T)$ be the descent of $T$ defined as the smallest non negative integer $q$ such that $T^{q}(X)=T^{q+1}(X)$. If no such integer exists, then we set $q(T)=\infty$.  By \cite[Theorem 3.3]{1} if both $p(T)$ and $q(T)$ are finite then $p(T)=q(T)$. It is well known that  $\lambda$ is a pole of resolvent of $T$ if and only if $0 < p(\lambda I-T)=q(\lambda I-T)<\infty$. An operator $T \in B(X)$ is said to be left Drazin invertible if $p(T)<\infty$ and $T^{p+1}(X)$ is closed.  We say that $\lambda \in \sigma_{a}(T)$ is a left pole of $T$ if $\lambda I-T$ is left Drazin invertible. An operator  $T \in \mathcal{B}(X)$ is said to be Drazin invertible if $p(T)=q(T)< \infty$. The \emph{left Drazin spectrum} is defined by  $$\sigma_{ld}(T):=\{\lambda \in \mathbb{C}: \lambda I-T \thinspace \mbox{is not left Drazin invertible}\}.$$
The \emph{Drazin  spectrum} is defined by  $$\sigma_{d}(T):=\{\lambda \in \mathbb{C}: \lambda I-T \thinspace \mbox{is not  Drazin invertible}\}.$$Clearly, $\sigma_{ld}(T)\subset \sigma_{d}(T)$. The class of all upper semi-Browder operators (resp. lower semi-Browder operators) is defined by $B_{+}(X):=\{T\in\phi_{+} (X): p(T) < \infty\}$ (resp. $B_{+}(X):=\{T\in\phi_{+} (X): q(T) < \infty\})$.  The class of all Browder operators is defined by $B(X):=B_{+}(X) \cap B_{-}(X)$. The \emph{upper semi Browder spectrum} is defined by  $$\sigma_{ub}(T):=\{\lambda \in \mathbb{C}: \lambda I-T \notin B_{+}(X)\},$$
the \emph{lower semi-Browder spectrum} is defined by  $$\sigma_{lb}(T):=\{\lambda \in \mathbb{C}: \lambda I-T \notin B_{-}(X)\}$$
and the \emph{Browder spectrum} is defined by  $$\sigma_b(T):=\{\lambda \in \mathbb{C}: \lambda I-T \notin B(X)\}.$$
For a bounded linear operator $T$  and a non negative integer $n$, define $T_{[n]}$ to be the restriction of $T$ to $T^n(X)$. If for some non negative integer $n$ the range space $T^n(X)$ is closed and $T_{[n]}$ is Fredholm (resp. an upper semi B-Fredholm, a lower semi B-Fredholm) then $T$ is said to be B-Fredholm (resp. an upper semi B-Fredholm, a lower semi B-Fredholm). In this case (see \cite{9}), the index of $T$ is defined as index of $T_{[n]}$. An operator $T \in B(X)$ is called   an upper semi B-Weyl (resp. a lower semi B-Weyl) if it is an upper semi B-Fredholm (resp. a lower semi B-Weyl)  having ind $(T)\leq 0$ (resp. ind $(T) \geq 0$). An operator $T \in B(X)$ is called B-Weyl  operator if it is B-Fredholm  having ind $(T)=0$. The \emph{upper semi B-Weyl spectrum} is defined by  $$\sigma_{usbw}(T):=\{\lambda \in \mathbb{C}: \lambda I-T \thinspace \mbox{is not upper semi B-Weyl}\}$$
 and the \emph{B-Weyl spectrum} is defined by  $$\sigma_{bw}(T):=\{\lambda \in \mathbb{C}: \lambda I-T \thinspace \mbox{is not B-Weyl}\}.$$
 An operator $T \in B(X)$ is said to have the single-valued extension property  at $\lambda_{0} \in \mathbb{C}$, abbreviated $T$ has SVEP at $\lambda_{0}$, if for every neighborhood $U$ of $\lambda_{0}$ the only analytic function $f:U \rightarrow X $ which satisfies the equation $(\lambda I-T)f(\lambda)=0$ is the function $f=0$. An operator $T$ is said to have SVEP if $T$ has SVEP at every $\lambda \in \mathbb{C}$. It is known that $$p(\lambda I-T)<\infty \Rightarrow T \thinspace \emph{has SVEP at}\thinspace\thinspace \lambda$$ and $$q(\lambda I-T)<\infty \Rightarrow T^* \thinspace \emph{has SVEP at}\thinspace\thinspace \lambda.$$ \\
 The quasi-nilpotent part of $T$, defined as$$H_{0}(T):=\{x\in X:\lim\limits_{n \rightarrow \infty} \Vert T^{n} x \Vert^{\frac{1}{n}}=0\}.$$
 Clearly, $H_0(T)$ is $T$-invariant subspace. 
 Let  $\Pi(T)=\sigma(T) \setminus \sigma_{d}(T)$ denote the set of all poles of $T$ and let $\Pi^{a}(T)=\sigma_a(T) \setminus \sigma_{ld}(T)$ denote the set of all left poles of $T$. Clearly, $\Pi(T) \subset \Pi^{a}(T)$. An operator $T \in B(X)$ is said to be polaroid if every isolated point of $\sigma (T)$ is a pole of the resolvent of $T$ and said to be $a$-polaroid operator if every isolated point of $\sigma_a (T)$ is a pole of the resolvent of $T$. Clearly, $$T \emph{\thinspace is \thinspace a-polaroid}  \Rightarrow T \emph{\thinspace is\thinspace polaroid.}$$
 A part of an operator is its restriction to a closed invariant subspace. An operator $T \in B(X)$ is hereditarily polaroid if every part of $T$ is polaroid. Clearly, every hereditarily polaroid is polaroid.\\ 
  Let $\bigtriangleup(T) = \sigma(T)\setminus \sigma_{w}(T)$, 
 $\bigtriangleup_{a}(T) = \sigma_{a}(T)\setminus \sigma_{uw}(T).$ We say  that $T$ satisfies Weyl's theorem if $\bigtriangleup(T)=\Pi_{00}(T)$, where $\Pi_{00}(T):=\{\lambda \in \mathbb{C}:\lambda \in$ iso $\sigma(T)$ and $0< \alpha (\lambda I-T) < \infty\}$ and that $T$ satisfies a-Weyl's theorem if $\bigtriangleup_a(T)=\Pi_{00}^{a}(T)$ where $\Pi_{00}^{a}(T):=\{\lambda \in \mathbb{C}:\lambda \in \mbox{iso}\sigma_{a}(T)\thinspace \mbox{and}\thinspace 0< \alpha    (\lambda I-T) < \infty\}.$ It is well known that  a-Weyl's theorem implies Weyl's theorem. 	
 Let $\Pi_{0}(T)=\sigma(T)\setminus \sigma_{b}(T)$ and $\Pi^{a}_{0}(T)=\sigma_{a}(T)\setminus \sigma_{ub}(T)$. We say that Browder's theorem holds for $T\in B(X)$ if $\bigtriangleup(T)=\Pi_{0}(T)$ and that a-Browder's theorem holds for $T\in B(X)$ if  $\bigtriangleup_{a}(T)=\Pi^{a}_{0}(T) $. It is well known that $a$-Browder's theorem implies Browder's theorem. Following Berkani \cite{10} we say that generalized Browder's theorem holds for $T\in B(X)$ if $\sigma_{bw}(T)=\sigma_{d}(T)$.  It is proved in \cite[Theorem 2.1]{8} that generalized Browder's theorem is equivalent to  Browder's theorem. It is said that generalized $a$-Browder's theorem holds for $T\in B(X)$ if $\sigma_{usbw}(T)=\sigma_{ld}$. It is proved in \cite[Theorem 2.2]{8} that generalized $a$-Browder's theorem is equivalent to  $a$-Browder's theorem.
 Following Rako$\check{\mbox{c}}$evic$\acute{\mbox{c}}$ \cite{20} $T \in B(X)$ is said to possess property $(w)$ if  $\bigtriangleup_{a}(T)=\Pi_{00}(T)$. It is well known that property $(w)$ implies a-Browder's theorem.
 
 Following Aiena et.al. \cite{4} we say that $T \in B(X)$ possesses property ($R)$ if $\Pi_{0}(T)=\Pi_{00}(T)$. According to Rashid \cite{16},  property ($S)$ holds for $T \in B(X)$ if $\sigma(T) \setminus \sigma_{b}(T)=\Pi_{00}(T)$. It is shown in \cite[Theorem 2.10]{16} that property $(R)$ implies property $(S)$ but the converse is not true in general. According to Gupta and Kashyap \cite{14}, property ($Bw)$ holds for $T \in B(X)$ if $\sigma(T)\setminus \sigma_{bw}(T)=\Pi_{00}(T)$. Recall \cite{17} that property ($Bgw)$ holds for $T \in B(X)$ if $\sigma_{a}(T)\setminus \sigma_{usbw}(T)=\Pi_{00}(T)$. It is shown in \cite[Theorem 2.16]{17}, that property $(Bgw)$ implies property $(Bw)$ but the converse is not true in general. Following Zariouh \cite{18} we say that property $(Z_{\Pi a})$ holds for $T$ if $\bigtriangleup(T)=\Pi^{a}(T)$.
 
In this paper, we introduce properties $(BR)$ and $(BgR)$. These properties are studied in connection with Weyl type theorems. We have shown property $(BgR)$ imples property $(BR)$ but the converse is not true in general. Property $(BgR)$ and property $(BR)$ are related to property $(R)$ and property $(S)$, respectively. Also, these properties are investigated in the framework of polaroid type  operators. In the last section, we have examined the stability of these properties under perturbations by   commuting finite-dimensional, quasi-nilpotent, Riesz and algebraic operators. 
  \section{Main Results}
\begin{definition}
\emph{A bounded linear operator $T$ is said to satisfy property $(BgR)$ if $\Pi^{a}(T)=\Pi_{00}(T)$ and is said to satisfy property $(BR)$ if $\Pi(T)=\Pi_{00}(T)$.}
\end{definition}
\begin{theorem}\label{theorem1}
Let $T \in B(X)$. If $T$ satisfies property $(BgR)$, then $T$ satisfies property $(BR)$.
\end{theorem}
\begin{proof}
Suppose that $T$ satisfies property $(BgR)$, then $\Pi^{a}(T)=\Pi_{00}(T)$. Let $\lambda \in  \Pi(T)$. Since $\Pi(T) \subset \Pi^{a}(T)$, $\lambda \in \Pi^{a}(T)= \Pi_{00}(T)$. Now, let $\lambda \in \Pi_{00}(T)$ which implies that $\lambda \in$ iso $\sigma(T) $ and $0< \alpha(\lambda I-T) < \infty$. As $T$ satisfies property $(BgR)$, $\lambda \in \Pi^{a}(T)$. Therefore, $\lambda \notin \sigma_{ld}(T)$. By \cite[Lemma 2.4]{11}, $ \lambda \notin \sigma_{ub}(T)$. Hence, by \cite[Corollary 3.21]{1} we get $\lambda \in \Pi(T)$. 
\end{proof}
The converse of Theorem \ref{theorem1} does not hold in general as shown by the following example:
\begin{example}\label{example1}
\emph{Let $R: l^2(\mathbb{N}) \rightarrow l^2(\mathbb{N})$ be the classical unilateral right shift and $Q: l^2(\mathbb{N}) \rightarrow l^2(\mathbb{N})$ be defined as $Q(x_1,x_2, \ldots)=(\frac{1}{2} x_1, x_2, \ldots$). If $T:=R \oplus Q$. Then $\sigma(T)= D$, where $D$ is the  closed unit disk and $\sigma_{a}(T)= S^{1} \cup \{ \frac{1}{2}\}$, where $S^{1}$ is the unit circle. Moreover, $\Pi_{00}(T)=\Pi(T)=\emptyset$ and $\Pi^{a}(T)=\{\frac{1}{2}\}.$ Therefore, $T$ satisfies property $(BR)$ but does not satisfy property $(BgR)$.}
\end{example}
\begin{lemma}\label{lemma1}
Let $T \in B(X)$, then $T^{*}$ has SVEP at $\lambda \notin \sigma_{ld}(T) $ if and only if $\sigma_{ld}(T)=\sigma_{d}(T).$ 
\end{lemma} 
\begin{proof}
Suppose that $T^{*}$ has SVEP at $\lambda \notin \sigma_{ld}(T) $. First we prove that $\sigma_{a}(T)=\sigma(T)$. Let $\lambda \notin \sigma_{a}(T) $ and $\lambda \in \sigma(T)$ which implies that $\lambda \notin \sigma_{ld}(T)$. Therefore, $p(\lambda I-T) < \infty$. Thus, $T$ has SVEP at $\lambda$. Also, $T^{*}$ has SVEP at $\lambda$. Thus, by \cite[Corollary 3.21]{1} we get $\lambda \in \Pi(T) \subset \sigma_{a}(T)$, a contradiction. Hence, $\sigma_{a}(T)=\sigma(T)$. Let $\lambda \notin \sigma_{ld}(T).$ As $\sigma_{ld}(T)=\sigma_{usbw}(T)$ $ \cup$ acc $\sigma_{a}(T)$, $\lambda \notin$ acc$\sigma_{a}(T)$ and $\lambda \notin \sigma_{usbw}(T)$. Since $\lambda \notin \sigma_{usbw}(T)$, $(\lambda I-T)$ is an upper semi B-Fredholm operator and ind $(\lambda I-T) \leq  0.$  Therefore, by \cite[Corollary 2.8]{23} we get $\lambda \notin \sigma_{bw}(T)$. As $\sigma_{a}(T)=\sigma(T)$, $\lambda \notin $ acc$\sigma(T)$ and it is known that $\sigma_{d}(T)=\sigma_{bw}(T)$ $\cup $ acc $\sigma(T)$ which gives $\lambda \notin \sigma_{d}(T).$ Hence, $\sigma_{ld}(T)=\sigma_{d}(T)$.

Conversely, let $\lambda \notin \sigma_{ld}(T)$. As  $\sigma_{ld}(T)=\sigma_{d}(T)$, $\lambda \notin \sigma_{d}(T)$ which implies that $q(\lambda I-T) < \infty$. Therefore, $T^{*}$ has SVEP at $\lambda$.
\end{proof}
\begin{theorem}\label{theorem2}
Let $T \in B(X)$. If $T$ satisfies property $(BR)$ and $T^{*}$ has SVEP at $\lambda \notin \sigma_{ld}(T) $, then $T$ satisfies property $(BgR)$.
\begin{proof}
By Lemma \ref{lemma1} $\sigma_{a}(T)=\sigma(T)$ and $\sigma_{ld}(T)=\sigma_{d}(T)$. Since  $T$ satisfies property $(BR)$  which gives $T$ satisfies property $(BgR)$.
\end{proof}
\end{theorem}
Recall that $T\in B(X)$ is said to be isoloid if every isolated point of $\sigma(T)$ is an eigenvalue of $T$. Clearly, every polaroid operator is isoloid. An operator $T\in B(X)$ is said to be finite-isoloid if every isolated point is an eigen value of finite multiplicity, i.e. $0 < \alpha(\lambda I-T) < \infty$.
\begin{theorem}\label{theorem3}
Let $T \in B(X)$, then
 
(i) If $T$ is polaroid and finite-isoloid, then $T$ satisfies property $(BR)$.

(ii) If $T$ is $a$-polaroid and finite-isoloid, then $T$ satisfies property $(BgR)$.
\end{theorem}
\begin{proof}
(i) Let $\lambda \in  \Pi(T)$. Since $T$ is finite-isoloid, $\lambda \in \Pi_{00}(T)$. Now, let $\lambda \in \Pi_{00}(T) \subset$ iso $\sigma(T)$. As $T$ is polaroid, $\lambda \in \Pi(T)$.\\
(ii) Let $\lambda \in \Pi^{a}(T)$ which implies that $\lambda \in \sigma_{a}(T)$ and $\lambda \notin \sigma_{ld}(T)$. This gives $\lambda \in$ iso$\sigma_{a}(T)$. As $T$ is $a$-polaroid and finite-isoloid, $\lambda \in \Pi_{00}(T)$. Now, let $\lambda \in \Pi_{00}(T)$ which  gives $\lambda \in$ iso $\sigma_{a}(T)$. Since $T$ is $a$-polaroid, $\lambda \in \Pi(T) \subset \Pi^{a}(T)$.
\end{proof}
\begin{theorem}\label{theorem4}
Let $T \in B(X)$. If $T$ satisfies property $(BgR)$, then $T$ satisfies property ($R$).
\end{theorem}
 \begin{proof}
  It suffices to prove that $\sigma_{ld}(T)=\sigma_{ub}(T)$. Let $\lambda \notin \sigma_{ld}(T)$. If $\lambda \in \sigma_{a}(T)$, then $\lambda \in \Pi^{a}(T)= \Pi_{00}(T)$. Therefore, by \cite[Lemma 2.4]{11} $\lambda \notin \sigma_{ub}(T)$. If $\lambda \notin \sigma_{a}(T)$, then $\lambda \notin \sigma_{ub}(T)$. 
 \end{proof}
  The converse of the above theorem is not true in general. This can be illustrated by the following example:
\begin{example}\label{example2}
\emph{Let $T: l^2(\mathbb{N}) \rightarrow l^2(\mathbb{N})$ be defined as $T(x_1,x_2, \ldots)=(0,\frac{-x_1}{2},0,0,\\ \ldots$).  Then $T^2 =0$. Therefore, $T$ is a nilpotent operator. Since $\alpha(T) $ is infinite, $\sigma_a (T)= \sigma_{uw} (T)=\sigma_{ub} (T)=\{0\}$ and $\sigma _{ld} (T)=\emptyset$. Thus, $\Pi_{0}^{a}(T)=\Pi_{00}(T)=\emptyset$ and $\Pi^{a}(T)=\{0\}.$ Hence, $T$ satisfies property ($R$) but $T$ does not satisfy property $(BgR)$.}
\end{example}
\begin{theorem}\label{theorem5}
Let $T \in B(X)$, then the following statements are equivalent:

(i) $T$ satisfies property $(Bgw)$,

(ii) $T$ satisfies property $(BgR)$ and generalized $a$-Browder's theorem,

(iii) $T$ satisfies property (w) and $\sigma_{ld}(T)=\sigma_{ub}(T)$.
\end{theorem}
\begin{proof}
(i) $\Leftrightarrow$ (ii)  Suppose that $T$ satisfies property ($Bgw$). Let $\lambda \in \sigma_{a}(T) \setminus \sigma_{usbw}(T)= \Pi_{00}(T)$. This implies that $\lambda \in$ iso $\sigma_{a}(T)$. It is known that $\sigma_{ld}(T)=\sigma_{usbw}(T)$ $\cup$ acc $\sigma_{a}(T)$ which gives $\lambda \notin \sigma_{ld}(T).$ Therefore, $T$ satisfies generalized $a$- Browder's theorem which gives $T$ satisfies property $(BgR)$. Conversely, let $T$ satisfies property $(BgR)$. Since $T$ satisfies generalized $a$- Browder's theorem. Therefore, $T$ satisfies property $(Bgw)$.\\
(ii) $\Leftrightarrow$ (iii) Suppose that $T$ satisfies property $(BgR)$ and generalized $a$-Browder's theorem. By Theorem \ref{theorem4} $T$ satisfies property $(R)$. This gives $\sigma_{ld}(T)= \sigma_{ub}(T)$. Let $\lambda \in \bigtriangleup_{a}(T)$, then $\lambda \in \sigma_{a}(T)$ and $\lambda \notin \sigma_{uw}(T)$. Since $T$ satisfies generalized $a$-Browder's theorem and we know that generalized $a$-Browder's theorem is equivalent to $a$-Browder's theorem, $\lambda \notin \sigma_{ub}(T)$. This gives $\lambda \in \Pi^a(T)= \Pi_{00}(T)$. Now, let $\lambda \in \Pi_{00}(T)=\Pi^a(T)$ which implies that $\lambda \notin \sigma_{ld}(T)$. As $\sigma_{ld}(T)=\sigma_{ub}(T),$ $\lambda \notin \sigma_{ub}(T)$. This implies that $\lambda \notin \sigma_{uw}(T)$. Therefore, $\lambda \in \bigtriangleup_{a}(T)$. Conversely, let $T$ satisfies property $(w)$. We know that property $(w)$ implies $a$- Browder's theorem which gives $\sigma_{uw}(T)=\sigma_{ub}(T)=\sigma_{ld}(T)$. Therefore, $\Pi^{a}(T)=\Pi_{00}(T)$.
\end{proof}
\begin{theorem}\label{theorem6}
Let $T \in B(X)$. If $T$ satisfies property $(BR)$, then $T$ satisfies property ($S$).
\end{theorem}
\begin{proof}
It suffices to prove that $\sigma_{d}(T)=\sigma_{b}(T)$. Let $\lambda \notin \sigma_{d}(T)$. If $\lambda \in \sigma(T)$, then $\lambda \in \Pi(T)= \Pi_{00}(T)$. Therefore, by \cite[Corollary 3.21]{1} we get $\lambda \notin \sigma_{b}(T)$. If $\lambda \notin \sigma(T)$, then $\lambda \notin \sigma_{b}(T)$.
\end{proof}
The converse of the above theorem is not true in general. Let $T$ be defined as in Example \ref{example2}. Then $\sigma(T)=\sigma_b(T)=\{0\}$ and $\sigma_{d}(T)=\Pi_{00}(T)=\emptyset$. Therefore, $T$ satisfies property $(S)$ but does not satisfy property $(BR)$.
\begin{theorem}\label{theorem7}
Let $T \in B(X)$, then the following statements are equivalent:

(i) $T$ satisfies property $(Bw)$,

(ii) $T$ satisfies property $(BR)$ and generalized Browder's theorem,

(iii) $T$ satisfies Weyl's theorem and $\sigma_{d}(T)=\sigma_{b}(T)$.
\end{theorem}
\begin{proof}
(i) $\Leftrightarrow$ (ii)  Suppose that $T$ satisfies property ($Bw$). Let $\lambda \in \sigma(T) \setminus \sigma_{bw}(T)= \Pi_{00}(T)$. This implies that $\lambda \in$ iso $\sigma(T)$. It is known that $\sigma_{d}(T)=\sigma_{bw}(T)$ $\cup$ acc $\sigma(T)$ which gives $\lambda \notin \sigma_{d}(T).$ Therefore, $T$ satisfies generalized Browder's theorem which gives $T$ satisfies property $(BR)$. Conversely, let $T$ satisfies property $(BR)$. Since $T$ satisfies generalized Browder's theorem. Therefore, $T$ satisfies property $(Bw)$.\\
(ii) $\Leftrightarrow$ (iii) Suppose that $T$ satisfies property $(BR)$ and generalized Browder's theorem. By Theorem \ref{theorem6} $T$ satisfies property $(S)$. This gives $\sigma_{d}(T)= \sigma_{b}(T)$. Let $\lambda \in \bigtriangleup(T)$, then $\lambda \in \sigma(T)$ and $\lambda \notin \sigma_{w}(T)$. Since $T$ satisfies generalized Browder's theorem and we know that generalized  Browder's theorem is equivalent to Browder's theorem, $\lambda \notin \sigma_{b}(T)$. This gives $\lambda \in \Pi(T)= \Pi_{00}(T)$. Now, let $\lambda \in \Pi_{00}(T)=\Pi(T)$ which implies that $\lambda \notin \sigma_{d}(T)$. As $\sigma_{d}(T)=\sigma_{b}(T),$ $\lambda \notin \sigma_{b}(T)$. This implies that $\lambda \notin \sigma_{w}(T)$. Therefore, $\lambda \in \bigtriangleup(T)$. Conversely, let $T$ satisfies Weyl's theorem. We know that Weyl's theorem implies Browder's theorem which gives $\sigma_{w}(T)=\sigma_{b}(T)=\sigma_{d}(T)$. Therefore, $\Pi(T)=\Pi_{00}(T)$.
\end{proof}
\begin{theorem}\label{theorem8}
Let $T \in B(X)$. If $T$ is $a$-polaroid, then property $(BgR)$ and property $(BR)$ are equivalent.
\end{theorem}
\begin{proof}
 It suffices to prove that $\Pi^{a}(T)\subset\Pi(T)$. Let $\lambda \in \Pi^{a}(T)$ which implies that $\lambda \in$ iso$\sigma_{a}(T)$. Since $T$ is $a$-polaroid, $\lambda \in \Pi(T)$. 
\end{proof}
\begin{theorem}\label{theorem9}
Let $T \in B(X)$. If $T$ is polaroid, then $T$ satisfies property $(Z_{\Pi a})$ if and only if $T$ satisfies property $(BgR)$ and Browder's theorem.
\end{theorem}
\begin{proof}
Suppose that $T$ satisfies  property $(Z_{\Pi a})$, then $\Pi^{a}(T)=\bigtriangleup(T)$. Let $\lambda \in \bigtriangleup(T)=\Pi^{a}(T)$. This gives $\alpha(\lambda I-T)=\beta(\lambda I-T) < \infty$ and $p(\lambda I-T) < \infty$. Therefore, by \cite[Theorem 3.4]{1} $q(\lambda I-T) < \infty$. Thus, $\lambda \in  \Pi_{0}(T)$. Hence, $T$ satisfies Browder's Theorem. Let $\lambda \in \Pi^{a}(T)$. As  $\Pi^{a}(T)=\bigtriangleup(T)= \Pi_{0}(T)$, $\lambda \in \Pi_{00}(T)$. Now, let  $\lambda \in \Pi_{00}(T)$. Since $T$ is polaroid, $\lambda \in \Pi(T) \subset \Pi^{a}(T)$.

Conversely, suppose that $T$ satisfies property $(BgR)$ and Browder's theorem. By Theorem \ref{theorem1} $T$ satisfies property $(BR)$. By Theorem \ref{theorem7} $T$ satisfies Weyl's theorem, then $\bigtriangleup(T)=\Pi_{00}(T)=\Pi^{a}(T)$. Hence, $T$ satisfies  property $(Z_{\Pi a})$.
\end{proof}
\section{Property $(BR)$ and property $(BgR)$ under perturbations}
In this section we study the stability of properties $(BR)$ and $(BgR)$ under perturbations. Recall that an operator $K \in B(X)$ is said to be a Reisz operator if $\lambda I-K \in \phi (X)$ for all $\lambda \neq 0$. All compact operators and quasi-nilpotent operators are examples of Reisz operators. We have  
$$ \sigma_* (T)= \sigma_* (T+K)$$
for every Reisz operator $K$ commuting with $T \in B(X)$ where $\sigma_*= \sigma_{uw}$ or $\sigma_{ub}$ or $\sigma_b$ or $\sigma_w$. Recall that for every quasi-nilpotent operator commuting with $T$, we have $\sigma(T+Q)= \sigma(T)$ and $\sigma_{a}(T+Q)= \sigma_{a}(T).$
\begin{theorem}\label{theorem10}
Let $T \in B(X)$ and $N$ be a nilpotent operator on $X$ commuting with $T$, then

(i)   $T$ satisfies property $(BR)$ if and only if $T+N$ satisfies property $(BR)$.

(ii) $T$ satisfies property $(BgR)$ if and only if $T+N$ satisfies property $(BgR)$.
\end{theorem}
\begin{proof}
(i)  By \cite[Lemma 2.2]{12} we have $\Pi(T+N)= \Pi(T)$ and by \cite[Theorem 3.5]{13} we know that $\Pi_{00}(T+N)=\Pi_{00}(T)$.  Hence, $T$ satisfies property $(BR)$ if and only if $T+N$ satisfies property $(BR)$.
\\
(ii)   By \cite[Corollay 3.8]{19} we have  $\Pi^{a}(T+N)= \Pi^{a}(T)$. Also, $\Pi_{00}(T+N)=\Pi_{00}(T)$. Hence, $T$ satisfies property $(BgR)$ if and only if $T+N$ satisfies property $(BgR)$.
\end{proof}
Let $Q$ be quasi-nilpotent operator  commuting with $T$ and $T$ satisfies property $(BgR)$ (property $(BR)$) then it is not necessary that $T+Q$ satisfies property $(BgR)$ (property $(BR)$).
\begin{example}
\emph{Let $T: l^2(\mathbb{N}) \rightarrow l^2(\mathbb{N})$ be defined as $T(x_1,x_2, \ldots)=\Big(0,\frac{x_1}{2},\frac{x_1}{3},\\ \frac{x_1}{4},\ldots\Big)$. Then $\sigma(T)=\sigma_{d}(T)=\{0\}$ and $ \sigma_{a}(T)=\sigma_{ld}(T)=\Pi_{00}(T)= \emptyset$. Therefore, $T$ is a quasi-nilpotent operator satisfying property $(BgR)$ and property $(BR)$. If we take $Q=-T$, then $T+Q$ is zero operator and we know that $\Pi(0)=\Pi^{a}(0)=\{0\}$ and $\Pi_{00}(0)=\emptyset.$ Hence, $T+Q$ does not satisfy property $(BgR)$ and property $(BR)$.}
\end{example}

\begin{lemma}\label{lemma2}
Let $T \in B(X)$ and $Q$ be a quasi-nilpotent operator commuting with $T$. If $T$ is $a$-polaroid satisfying property $(BgR)$, then $T+Q$ is $a$-polaroid operator.
\end{lemma}
\begin{proof}
Let $\lambda \in$ iso $\sigma_{a}(T+Q)$, then $\lambda \in$ iso $\sigma_{a}(T).$ As $T$ is  $a$-polaroid operator, $\lambda \in \Pi(T)$. By Theorem \ref{theorem8} $T$ satisfies property $(BR)$ which implies that $\lambda \in \Pi_{00}(T)$. This gives $\lambda \notin \sigma_{b}(T)$.  Therefore, $\lambda \notin \sigma_{b}(T+Q)$ which implies that $\lambda \notin \sigma_{d}(T+Q)$. Hence, $\lambda \in \Pi (T+Q)$.
\end{proof}
\begin{theorem}\label{theorem11}
Let $T \in B(X)$ and $Q$ be a quasi-nilpotent operator commuting with $T$. If $T$ is $a$-polaroid satisfying $a$-Browder's theorem, then $T$ satisfies property $(BgR)$ if and only if $T+Q$ satisfies property $(BgR)$.
\end{theorem}
\begin{proof}
By Lemma \ref{lemma2} $T+Q$ is $a$-Polaroid, then $\Pi_{00}(T+Q) \subset \Pi^{a}(T+Q).$ Let $\lambda \in \Pi^{a}(T+Q)$ which implies that $\lambda \in$ iso $\sigma_{a}(T+Q)$=iso $\sigma_{a}(T)$. Since $T$ is $a$-Polaroid operator, $\lambda \in \Pi(T) \subset$ iso $\sigma(T)=$ iso $\sigma(T+Q)$. As $T$ satisfies property $(BgR)$,  $\lambda \in \Pi_{00}(T)$. By Theorem \ref{theorem5} $T$ satisfies property $(w)$, then $\Pi_{00}(T)= \bigtriangleup_{a}(T)$. Therefore, $\lambda \notin \sigma_{uw}(T)=\sigma_{uw}(T+Q)$. Thus, $0 < \alpha(\lambda I-(T+Q)) < \infty$. Hence, $\lambda \in \Pi_{00}(T+Q)$. 
\end{proof}
\begin{theorem}\label{{theorem12}}
Let $T \in B(X)$ and $Q$ be a quasi-nilpotent operator commuting with $T$. If $T$ is $a$-polaroid and finite-isoloid, then $T+Q$ satisfies property $(BgR)$.
\end{theorem}
\begin{proof}
By Theorem \ref{theorem3} $T$ satisfies property $(BgR)$. Therefore, by Lemma \ref{lemma2} $T+Q$ is $a$-polaroid which implies that  $\Pi_{00}(T+Q) \subset \Pi^{a}(T+Q).$ Let $\lambda \in \Pi^{a}(T+Q)$ then $\lambda \in$ iso $\sigma_{a}(T+Q)=$ iso $\sigma_{a}(T).$  As $T$ is  $a$-polaroid, $\lambda \in \Pi^{a}(T)=\Pi_{00}(T)$. Therefore,  $\lambda \notin \sigma_{b}(T)=\sigma_{b}(T+Q)$. This gives  $\lambda \in \Pi_{00}(T+Q)$.
\end{proof}
\begin{theorem}\label{theorem13}
Let $T \in B(X)$ and $F$ be a finite-dimensional operator commuting with $T$. If $T$ is polaroid  satisfying property $(BR)$, then $T+F$ satisfies property $(BR)$. 
\end{theorem}
\begin{proof}
Let $\lambda \in \Pi(T+F),$ then  $\lambda \notin$ acc $\sigma(T+F)$.  By \cite[Theorem 4.8]{22} we have acc $\sigma(T+F)=$ acc $\sigma(T)$ therefore, $\lambda \notin$ acc $\sigma(T)$. If $\lambda \in \sigma(T)$, then $\lambda\in $ iso $\sigma(T)$. As $T$ is polaroid, $\lambda \in \Pi(T)=\Pi_{00}(T).$ Therefore, $\lambda \notin \sigma_{b}(T)=\sigma_{b}(T+F)$. If $\lambda \notin \sigma(T)$, then  $\lambda \notin \sigma_{b}(T)= \sigma_{b}(T+F)$. Therefore, in both the cases we get, $\lambda \in \Pi_{00}(T+F).$ Similarly, we can prove that $\Pi_{00}(T+F) \subset \Pi(T+F)$.
 \end{proof}
\begin{theorem}\label{theorem14}
Let $T \in B(X)$ and $F$ be a finite-dimensional operator commuting with $T$. Suppose that $T^{*}$ has SVEP at $\lambda \notin \sigma_{ld}(T) $. If $T$ is $a$-polaroid satisfying property $(BgR)$, then $T+F$ satisfies property $(BgR)$.
\end{theorem}
\begin{proof}
By Theorems \ref{theorem1} and \ref{theorem13}, $T+F$ satisfies property $(BR)$. By Theorem \ref{theorem8} it suffices to prove that $(T+F)$ is $a$-polaroid. Let $\lambda\in  $ iso $\sigma_{a}(T+F)$, then  $\lambda \notin $ acc $\sigma_{a}(T+F)=$ acc $\sigma_{a}(T).$ If $\lambda \in \sigma_{a}(T)$, then $\lambda\in  $ iso $\sigma_{a}(T)$. Since $T$ is $a$- Polaroid, $\lambda \in \Pi(T)=\Pi_{00}(T)$. Therefore, $\lambda \notin \sigma_{b}(T)=\sigma_{b}(T+F)$. If $\lambda \notin \sigma_{a}(T)$ and since $T^{*}$ has SVEP at $\lambda \notin \sigma_{ld}(T) $ then by the proof of Lemma \ref{lemma1} $\sigma(T)=\sigma_{a}(T)$ which gives $\lambda \notin \sigma(T)$. Therefore, $\lambda \notin \sigma_{b}(T)=\sigma_b(T+F)$.  Thus, in both the cases $\lambda$ is pole of resolvent of $T+F$.
\end{proof} 
Define 
$$\Pi_{0f}:=\{\lambda \in \mbox{iso} \thinspace \sigma(T):\alpha(\lambda I-T) < \infty\}.$$
\begin{lemma}\label{lemma3}
Let $T\in B(X)$ and  $K$ is a Riesz operator commuting with $T$, then $\Pi_{00}(T+K)$ $\cap$ $\sigma(T) \subset $  iso $\sigma(T).$
\end{lemma}
\begin{proof}
 By \cite[Lemma 3.3]{15} we have $\Pi_{00}(T+K) \cap \sigma(T)\subset \Pi_{0f}(T+K) \cap \sigma(T) \subset $ iso $\sigma(T).$
\end{proof}
\begin{theorem}\label{theorem15}
Let $T \in B(X)$ and $K$ be a Riesz operator commuting with $T$. If $T$ is polaroid satisfying property $(BR)$,  then $T+K$ satisfies property $(BR)$ if and only if $\sigma_{d}(T)\subset \sigma_{d}(T+K).$
\end{theorem}
\begin{proof}
Suppose that $T+K$ satisfies property $(BR)$. Let $\lambda \notin \sigma_{d}(T+K).$ If $\lambda \in \sigma(T+K)$, then $\lambda \in \Pi(T+K)=\Pi_{00}(T+K)$. This implies that $\lambda \notin \sigma_{b}(T+K)=\sigma_{b}(T)$. If $\lambda \notin \sigma(T+K)$, then $\lambda \notin \sigma_{b}(T+K)=\sigma_{b}(T)$. Therefore, in both the cases $\lambda \notin \sigma_{d}(T)$.

Conversely, suppose that $\lambda\in \Pi(T+K)$, then $\lambda \notin \sigma_{d}(T+K)$. As $\sigma_{d}(T)\subset \sigma_{d}(T+K)$, $\lambda \notin \sigma_{d}(T).$ If $\lambda \in \sigma(T)$, then $\lambda \in \Pi(T)= \Pi_{00}(T)$. Therefore, $\lambda \notin \sigma_{b}(T)=  \sigma_{b}(T+K)$. This gives $\lambda \in \Pi_{00}(T+K)$. If $\lambda \notin \sigma(T)$, then also $\lambda \in \Pi_{00}(T+K)$. Now let $\lambda \in \Pi_{00}(T+K)$. If $\lambda \in \sigma(T)$, then by Lemma \ref{lemma3} $\lambda \in$ iso $\sigma(T)$. Since $T$ is polaroid, $\lambda \in \Pi(T)=\Pi_{00}(T)$. This implies that $\lambda \notin \sigma_{b}(T)=\sigma_{b}(T+K).$ Therefore, $\lambda \in \Pi(T+K)$. If $\lambda\notin \sigma(T)$ then it is easily seen that $\lambda \in \Pi(T+K)$. 
\end{proof} 
\begin{theorem}\label{theorem16}
Let $T \in B(X)$ and $K$ be a Riesz operator commuting with $T$. Suppose that $T^{*}$ has SVEP at $\lambda \notin \sigma_{ld}(T) $. If $T$ is polaroid satisfying property  $(BgR)$, then $T+K$ satisfies property $(BgR)$ if and only if $\sigma_{ld}(T) \subset \sigma_{ld}(T+K)$.
\end{theorem}
\begin{proof}
Suppose that $T+K$ satisfies property $(BgR)$. Let $\lambda \notin \sigma_{ld}(T+K)$. If  $\lambda \in \sigma_{a}(T+K)$, then $\lambda \in \Pi^{a}(T+K)=\Pi_{00}(T+K)$. By Theorem \ref{theorem1} $T+K$ satisfies property $(BR)$. Therefore, $\lambda \in \Pi(T+K)$ which implies that $\lambda \notin \sigma_{b}(T+K)=\sigma_{b}(T)$. This gives $\lambda \notin \sigma_{ld}(T)$. If $\lambda \notin \sigma_{a}(T+K)$, then $\lambda \notin \sigma_{ub}(T+K)=\sigma_{ub}(T)$. Hence, $\lambda \notin \sigma_{ld}(T)$.

Conversely, suppose that $\sigma_{ld}(T) \subset \sigma_{ld}(T+K)$. First we prove that $\sigma_{d}(T)\subset \sigma_{d}(T+K).$ Let $\lambda \notin \sigma_{d}(T+K)$ which implies that $\lambda \notin \sigma_{ld}(T+K).$ This gives $\lambda \notin \sigma_{ld}(T).$ By Lemma \ref{lemma1} $\lambda \notin \sigma_{d}(T).$  Then by Theorems \ref{theorem1} and \ref{theorem15}, $(T+K)$ satisfies property $(BR)$. By Theorem \ref{theorem2} for $T+K$ to satisfy property $(BgR)$ it suffices to prove that $(T+K)^{*}$ has SVEP at $\lambda \notin \sigma_{ld}(T+K)$. Let $\lambda \notin \sigma_{ld}(T+K)$. Since $\sigma_{ld}(T) \subset \sigma_{ld}(T+K)$, $\lambda \notin \sigma_{ld}(T)$. Therefore, $T^{*}$ has SVEP at $\lambda$. Since $K$ is a Riesz operator, $K^{*}$ is also Riesz operator. By \cite[Theorem 0.3]{6} $(T+K)^{*}$ has SVEP at $\lambda$.
\end{proof}
Recall that $T \in B(X)$ is said to be algebraic if there exists a non-constant polynomial $h$ such that $h(T)=0$. Nilpotent operator is trivial example of algebraic operator. If for some $n \in \mathbb{N}$, $K^{n}$ is finite dimensional then $K$ is an algebraic operator. Every  algebraic operator has finite spectrum.
\begin{lemma}\label{lemma4}
Let $T \in B(X)$ and $N$ be nilpotent operator commuting with $T$ If $T$ is polaroid and finite-isoloid, then $T+N$ is polaroid and finite-isoloid.
\end{lemma}
\begin{proof}
Let $\lambda \in$ iso $\sigma(T+N)$. By \cite[Theorem 2.10]{3} $T+N$ is also polaroid. By Theorem \ref{theorem10} $T+N$ satisfies property $(BR)$.  This gives  $\lambda \in \Pi_{00}(T+N)$. Therefore, $0 < \alpha(\lambda I-T) < \infty.$ Hence, $T+N$ is finite-isoloid.
\end{proof}
Recall that $T \in B(X)$ hereditarily isoloid if every part of $T$ is  isoloid. We call $T \in B(X)$ hereditarily finite-isoloid if every part of $T$ is finite-isoloid. Obviously, hereditarily finite-isoloid is hereditarily isoloid.
\begin{theorem}\label{theorem17}
Let $T \in B(X)$ and $K$ be an algebraic operator commuting with $T$. Suppose that $T$ has SVEP. If $T$ is hereditarily polaroid and hereditarily finite-isoloid,  then $T+K$ satisfies property $(BR)$.
\end{theorem}
\begin{proof}
Since $K$ is algebraic, spectrum of $K$ is finite, then there exists an integer $n \geq 1$ such that $\sigma(T)=\{\lambda_{1},\lambda_{2},\ldots ,\lambda_{n}\}.$ Let $Y_{j}= P_{j}(X)$, where $P_{j}$ denotes the spectral projection associated with $K$ and the spectral sets $\{\lambda_{j}\},$ then $X= \mathop{\oplus}\limits_{j=1}^n  Y_{j}$. Since $K$ is algebraic, each $\lambda_{j}$ is a pole of resolvent of $K$. Therefore, for each j there exists an integer $p_{j} \geq 1$ such that $Y_{j}=H_{0}(\lambda_{j} I-K)=$ ker $(\lambda_{j} I-K)^{p_{j}}.$ Let $K_{j}:=K|Y_{j}$ and  $T_{j}:=T|Y_{j}$. Since $T$ and $K$ commutes, $T_{j}$ and $K_{j}$ commutes. Clearly,
$$T+K = \mathop{\oplus}\limits_{j=1}^n  (T_{j} +K_{j}).$$ This implies that $\sigma(T+K)=\mathop{\cup}\limits_{j=1}^n  \sigma(T_{j}+K_{j})$. Let $N_{j}= \lambda_{j} I-K_{j}$, then by \cite[Theorem 2.15]{2} $N_{j}$ is nilpotent.  Now we prove that $T+K$ is finite-isoloid. Let $\lambda \in$ iso $\sigma(T+K)$, then $\lambda \in $ iso $\sigma(T_{j}+K_{j})$ for some $j=1,2,\ldots ,n$. Therefore, $\lambda -\lambda_{j} \in $ iso $\sigma(T_{j}+K_{j}-\lambda_{j} I)$. As $T_{j}$ is polaroid and finite-isoloid, then by Lemma \ref{lemma4} $T_{j}+K_{j}-\lambda_{j} I$ is polaroid and finite-isoloid. Therefore, $ 0 < \alpha( (\lambda -\lambda_{j}) I-(T_{j}+K_{j}-\lambda_{j})) < \infty $ which  implies that $0 <\alpha(( \lambda  I-(T_{j}+K_{j}))<\infty$.  Since $$\mbox{ker}(( \lambda  I-(T+K))= \mathop{\oplus}\limits_{j=1}^n \mbox{ker} (\lambda I-(T_{j}+K_{j}))$$ we have  $0< \alpha((\lambda  I-(T+K))< \infty $  Thus, $T+K$ is finite-isoloid. By \cite[Theorem 2.15]{2} $T+K$ is polaroid. Hence, by Theorem \ref{theorem3} $T+K$ satisfies property $(BR)$.
\end{proof}
Let $\mathcal{H}_{nc}(\sigma(T))$ denote the set of all complex valued analytic functions defined on an open neighborhood of $\sigma(T)$, such that $f$ is injective and non constant on each of the components of its domain. For $f \in \mathcal{H}_{nc}(\sigma(T))$, define $f(T)$ using classical functional calculus.
\begin{lemma}\label{lemma5}
Let $T \in B(X)$. If $T$ is  polaroid and finite-isoloid then $f(T)$ is  polaroid and finite-isoloid for all $f \in \mathcal{H}_{nc}(\sigma(T)).$
\end{lemma}
\begin{proof}
By \cite[Lemma 3.11]{21} we know that $f(T)$ is  polaroid. Let $\lambda \in$ iso$\sigma(f(T))$. The proof of \cite[Lemma 3.11]{21} shows that $\lambda \in f$(iso $ \sigma(T))$. Therefore, there exists $\mu \in$ iso $ \sigma(T)$ such that $f(\mu)=\lambda$. Since $T$ is polaroid and finite isoloid, $0 < p(\lambda I-T)= q(\lambda I-T)< \infty$ and $0 < \alpha(\lambda I-T) < \infty$. This gives $\mu \notin \sigma_{b}(T).$ As $f$ is injective, $\lambda =f(\mu) \notin f(\sigma_{b}(T))$. It is  known that  $f(\sigma_{b}(T))= \sigma_{b}(f(T))$. This gives $\lambda \notin \sigma_{b}(f(T))$. Thus, $0 < \alpha(\lambda I-f(T)) < \infty$. Hence, $f(T)$ is finite isoloid.
\end{proof}
\begin{theorem} 
Let $T \in B(X)$ and $K$ be an algebraic operator commuting with $T$. Suppose that $T$ has SVEP. If $T$ is hereditarily polaroid and hereditarily finite-isoloid,  then $f(T+K)$ satisfies property $(BR)$ for all $f \in \mathcal{H}_{nc}(\sigma(T)).$ 
\end{theorem}
\begin{proof}
By Theorem \ref{theorem17} $T+K$ is polaroid and finite- isoloid. Therefore, by Lemma \ref{lemma5} $f(T+K)$ is  polaroid and finite-isoloid. Hence, by Theorem \ref{theorem3} $f(T+K)$ satisfies property $(BR)$.
\end{proof}

\textbf{Acknowledgement} The corresponding author is supported by Department of Science and Technology, New Delhi, India (Grant No. DST/INSPIRE Fellowship/[IF170390]).

\end{document}